\documentclass[a4paper,11pt]{amsart}

\usepackage{amsmath, amsthm, amsfonts, amssymb}
\usepackage[bookmarks=false]{hyperref}
\usepackage{enumerate}

\usepackage{color}

\numberwithin{equation}{section}

\newtheorem{letterthm}{Theorem}

\newtheorem{thm}{Theorem}[section]
\newtheorem{lem}[thm]{Lemma}

\newtheorem{prop}[thm]{Proposition}

\theoremstyle{definition}

\newtheorem*{definition}{Definition}

\newcommand{\rd}{\text{\rm d}}

\newcommand{\rL}{\mathord{\text{\rm L}}}

\newcommand{\ovt}{\mathbin{\overline{\otimes}}}

\newcommand{\II}{{\rm II}}
\newcommand{\III}{{\rm III}}

\begin{document}

\title{Strongly ergodic actions have local spectral gap}

\begin{abstract}
We show that an ergodic measure preserving action $\Gamma \curvearrowright (X,\mu)$ of a discrete group $\Gamma$ on a $\sigma$-finite measure space $(X,\mu)$ satisfies the local spectral gap property introduced in \cite{BISG15} if and only if it is strongly ergodic. In fact, we prove a more general local spectral gap criterion in arbitrary von Neumann algebras. Using this criterion, we also obtain a short proof of Connes' spectral gap theorem for full $\II_1$ factors \cite{Co75b} as well as its recent generalization to full type $\III$ factors \cite{Ma16}.
\end{abstract}

\address{Laboratoire de Math\'ematiques d'Orsay\\ Universit\'e Paris-Sud\\ CNRS\\ Universit\'e Paris-Saclay\\ 91405 Orsay\\ FRANCE}

\author{Amine Marrakchi}
\email{amine.marrakchi@math.u-psud.fr}

\thanks{The author is supported by ERC Starting Grant GAN 637601}



\maketitle

%%%%%%%%%%

\section{Introduction}

The main goal of this paper is to show that the local spectral gap property introduced in \cite{BISG15} is, in fact, equivalent to strong ergodicity for arbitrary ergodic measure preserving actions. Before we give a precise statement, let us first recall the notion of local spectral gap.

\begin{definition}
Let $\Gamma \curvearrowright (X,\mu)$ be  an ergodic measure preserving action of a discrete group $\Gamma$ on a $\sigma$-finite measure space $(X,\mu)$. Let $B \subset X$ be a measurable subset with $0 < \mu(B) < +\infty$. We say that $\Gamma \curvearrowright (X,\mu)$ has \emph{local spectral gap} with respect to $B$ if there exists a finite set $S \subset \Gamma$ and a constant $\kappa > 0$ such that
\[ \|f\|_{2,B} \leq \kappa \sum_{g \in S} \| g \cdot f-f \|_{2,B} \]
for every function $f \in \rL^2(X,\mu)$ such that $\int_B f \: \rd \mu =0$, where 
\[ \|f\|_{2,B} := \left( \int_B |f|^2 \: \rd \mu \right)^{1/2}. \]
\end{definition}
This local spectral gap property depends in general on the choice of the set $B$ (see \cite[Remark $1.3.(4)$]{BISG15}). When $\mu$ is a probability measure and $B=X$, we just get the usual notion of spectral gap for probability measure preserving actions.

It is well-known that, for probability measure preserving actions, spectral gap implies strong ergodicity while the converse is not true. However, our main result shows that the local spectral gap property introduced in \cite{BISG15} is equivalent to strong ergodicity, for any ergodic measure preserving action.
\begin{letterthm} \label{local_gap_action}
Let $\Gamma \curvearrowright (X,\mu)$ be an ergodic measure preserving action of a discrete group $\Gamma$ on a $\sigma$-finite measure space $(X,\mu)$. Then the following are equivalent:
\begin{enumerate}
\item The action $\Gamma \curvearrowright (X,\mu)$ is strongly ergodic.
\item There exists a subset $B \subset X$ with $0 < \mu(B) < +\infty$ such that the action $\Gamma \curvearrowright (X,\mu)$ has local spectral gap with respect to $B$.
\end{enumerate}
\end{letterthm}

Theorem \ref{local_gap_action} is deduced from the following more general local spectral gap criterion in tracial von Neumann algebras. For a non-tracial version, see Theorem \ref{local_gap_all}.

\begin{letterthm} \label{local_gap_finite}
Let $(M,\tau)$ be a tracial von Neumann algebra. Let $N \subset M$ be any von Neumann subalgebra and $\Sigma \subset M$ any self-adjoint subset. Suppose that for every bounded net $(x_i)_{i \in I}$ in $N$ which satisfies $\lim_i \|x_i a-ax_i \|_2 =0$ for all $a \in \Sigma$, we have $\lim_i \|x_i - \tau(x_i) \|_2 =0$.

Then we can find a non-zero projection $p \in N$, a finite subset $S \subset \Sigma$ and a constant $\kappa > 0$ such that for every $x \in pNp$ with $\tau(x)=0$, we have
\[  \|x\|_{2} \leq \kappa \sum_{ a \in S} \| x(pap) - (pap)  x \|_2.  \]
Moreover, we can chose $p$ to be arbitrarily close to $1$.
\end{letterthm}

We provide other applications of Theorem \ref{local_gap_finite} in the last section of this article.

\section{The local spectral gap property in arbitrary von Neumann algebras}

In this section, we prove Theorem \ref{local_gap_finite}. In fact, we prove the following more general version for arbitrary von Neumann algebras (not necessarily tracial). For any faithful normal state $\varphi$ on a von Neumann algebra $M$, we will say that a vector $\xi \in \rL^2(M)$ is \emph{$\varphi$-bounded} if $\xi \in M\xi_\varphi$ where $\xi_\varphi \in \rL^2(M)$ is the cyclic vector associated to $\varphi$.

\begin{thm} \label{local_gap_all}
Let $M$ be any von Neumann algebra with a faithful normal state $\varphi$. Let $N \subset M$ be any von Neumann subalgebra and $\Sigma=J\Sigma \subset \rL^2(M)$ any self-adjoint subset of $\varphi$-bounded elements. Suppose that for every bounded net $(x_i)_{i \in I}$ in $N$ which satisfies $\lim_i \| x_i \xi-\xi x_i \| = 0$ for all $\xi \in \Sigma$, we have $\lim_i \| x_i - \varphi(x_i) \|_\varphi =0$.

Then we can find a non-zero projection $p \in N$, a finite subset $S \subset \Sigma$ and a constant $\kappa > 0$ such that for every $x \in pNp$ with $\varphi(x)=0$ we have
\[  \|x\|_{\varphi} \leq \kappa \sum_{ \xi \in S} \| p(x\xi - \xi  x)p \|.  \]
Moreover, we can chose $p$ to be arbitrarily close to $1$.
\end{thm}

\begin{prop} \label{general_gap}
Let $M$ be any von Neumann algebra and let $\varphi$ be any normal state. Take $S \subset \rL^2(M)$ any finite self-adjoint subset of $\varphi$-bounded elements. Let $N$ any von Neumann subalgebra of $M$. Then the following are equivalent: 
\begin{enumerate}
\item  There exists a constant $\kappa > 0$ such that for any projection $p \in N$ we have
\[ \varphi(p)\varphi(1-p) \leq \kappa \sum_{\xi \in S} \| p \xi-\xi p \|^2 \]
\item There exists a constant $\kappa' > 0$ such that for all $x \in N$ we have
\[ \| x-\varphi(x) \|_\varphi \leq \kappa' \sum_{\xi \in S} \| x \xi - \xi x\| \]
\end{enumerate}
\end{prop}

The proof of Proposition \ref{general_gap} relies on the following lemma which is inspired by Namioka's trick. Our new input is item $(2)$.

\begin{lem} \label{spectral_resolution}
Let $M$ be any von Neumann algebra. For $a \geq 0$ and $x \in M^+$, we use the notation $E_a(x)=1_{[a,+\infty)}(x)$.
\begin{enumerate}
\item For every positive elements $x,y \in M^+$ and every $\xi \in \rL^2(M)$, we have
\[ \int_0^\infty \| E_a(x^2) \xi - \xi E_a(y^2) \|^2 \: \rd a \leq \| x \xi +\xi y \| \cdot \| x \xi- \xi y \| \]
\item For every positive elment $x \in M^+$ and every state $\varphi \in M^+_*$, we have 
\[ \| x-\varphi(x) \|_\varphi^2  \leq \int_0^\infty \varphi( E_{a}(x^{2})) \varphi \left(1-E_{a}(x^{2}) \right) \rd a \]
\end{enumerate}
\end{lem}
\begin{proof}
For item $(1)$, see the proof of \cite[Theorem 2]{CS78}. We will only prove $(2)$. First, note that $\| x- \varphi(x) \|_\varphi^2=\varphi(x^2)-\varphi(x)^2$ and that
\[ \varphi(x^2)=\int_0^\infty \varphi(E_a(x^2)) \: \rd a\]
Therefore, we only have to show that 
\[ \int_0^\infty \varphi( E_{a}(x^{2}))^2 \: \rd a  \leq \varphi(x)^2\]
On $M \ovt M$, we have $E_{a}(x^{2}) \otimes E_{a}(x^{2}) \leq E_{a}(x \otimes x)$. Hence, by appying $\varphi \otimes \varphi$ we get 
\[ \varphi( E_{a}(x^{2}))^2 \leq (\varphi \otimes \varphi)(E_a(x \otimes x)) \]
and thus, after integrating over $a$, we finally  get
\[ \int_0^\infty \varphi( E_{a}(x^{2}))^2 \: \rd a  \leq (\varphi \otimes \varphi)(x \otimes x)=\varphi(x)^2\]
\end{proof}

\begin{proof}[Proof of Proposition \ref{general_gap}]
$(2)$ clearly implies $(1)$ when applied to a projection.

Now assume that $(1)$ holds for some constant $\kappa > 0$ and let us show that $(2)$ also holds for some constant $\kappa'> 0$. A direct application of Lemma \ref{spectral_resolution} shows that for every element $x \in N^+$, we have
 \[ \|x-\varphi(x) \|_\varphi^2 \leq \kappa  \sum_{\xi \in S} \|x \xi +\xi x \|  \cdot \|x \xi-\xi x \| \]
Since the elements of $S$ are $\varphi$-bounded and $S$ is self-adjoint, we can find a constant $C > 0$ such that for all $x \in N^+$, we have $\|x\xi+\xi x \|\leq \|x\xi \|+\|\xi x  \| \leq C \|x\|_\varphi$. Hence, we get 
 \[ \forall x \in N^+, \quad \|x-\varphi(x) \|_\varphi^2 \leq C\kappa \|x\|_\varphi \sum_{\xi \in S}  \|x \xi-\xi x \| \]
Now, for every $x=x^* \in N$ with $\varphi(x)=0$, write $x=x_+-x_-$ where $x_+,x_- \in N^+$ and $x_+x_-=0$. Then we have
\[ \|x\|_\varphi^2 \leq 2(\|x_+-\varphi(x_+) \|_\varphi^2+\|x_--\varphi(x_-) \|_\varphi^2) \leq 2 C\kappa \|x\|_\varphi \sum_{\xi \in S} (\|x_+ \xi-\xi x_+ \|+\|x_-\xi -\xi x_- \|)\]
and since $\|x_{\pm}\xi -\xi  x_{\pm}\| \leq \|x \xi-\xi x \|$, we obtain
\[ \|x\|_\varphi \leq 4 C \kappa \sum_{\xi \in S} \|x \xi-\xi x \| \]
 By applying this inequality to $x-\varphi(x)$, we obtain $(2)$ for every self-adjoint $x \in N$. Finally, since $S$ is self-adjoint, it is easy to obtain $(2)$ for every element $x \in N$ by decomposing it into its real and imaginary part.
\end{proof}

\begin{proof}[Proof of Theorem \ref{local_gap_all}]
Fix $0 < \varepsilon < \frac{1}{4}$. Then, there exists a finite self-adjoint subset $S \subset \Sigma$ and $\eta > 0 $ such that for every projection $p \in N$ we have 
\[ \sum_{ \xi \in S} \|p \xi-\xi p \|_2^2 \leq \eta  \; \Longrightarrow \; \min(\varphi(p),\varphi(1-p)) \leq  \varepsilon \]
Consider the set $\Lambda$ of all projections $e$ in $N$ such that $\varphi(e) \leq \varepsilon$ and
\[ \sum_{ \xi \in S} \|e \xi-\xi e \|^2 \leq \eta \varphi(e) \]
Then $\Lambda$ is closed for the strong topology, hence it is an inductive set. Therefore, by Zorn's lemma, we can choose $e$ a maximal element of $\Lambda$. Let $p=1-e$. Take a projection $f \in pNp$. Suppose that
\[ \varphi(f)\varphi(p-f) > \frac{1}{\eta} \sum_{\xi \in S} \| f (p\xi p) - (p\xi p)f \|^2 \]
Then up to replacing $f$ by $p-f$, we can suppose that $\varphi(f) \leq \frac{1}{2}$. Now let $q=e+f$. Then we can check that
\[ \sum_{\xi \in S} \|q \xi - \xi q \|^2 \leq \sum_{ \xi \in S} \|e \xi-\xi e \|^2+\sum_{\xi \in S} \| f(p \xi p) - (p\xi p)f \|^2 \leq \eta\varphi(e) + \eta\varphi(f) = \eta \varphi(q)\]
But, since $e$ is maximal in $\Lambda$, we know that $q \notin \Lambda$. Hence we must have $\varphi(q) > \varepsilon$. Therefore, by the choice of $S$ and $\eta$, we must have $\varphi(q) \geq 1- \varepsilon$. But $\varphi(q)=\varphi(e)+\varphi(f) \leq \varepsilon + \frac{1}{2}$. Since $\varepsilon < \frac{1}{4}$, this is a contradiction.  Hence, for all projections $f \in pNp$, we have
\[ \varphi(f)\varphi(p-f) \leq  \frac{1}{\eta} \sum_{\xi \in S} \| f (p \xi p) - (p\xi  p)f \|^2 \]
Finally, for $\varphi'=\frac{1}{\varphi(p)}p\varphi p$ and $S'=pSp$, we can use Proposition \ref{general_gap} to conclude that there exists $\kappa' > 0$ such that
\[ \forall x \in p N p, \quad  \|x-\varphi'(x)\|_{\varphi'} \leq \kappa' \sum_{ \xi \in S} \| x(p\xi p)-(p\xi p) x \|.  \]
\end{proof}

\section{Applications}

We first prove our main theorem which motivated Theorem \ref{local_gap_finite}.

\begin{proof}[Proof of Theorem \ref{local_gap_action}] We only have to prove that $(1)$ implies $(3)$. Let $A \subset X$ be a subset with $\mu(A)=1$. Let $q=\mathbf{1}_A \in \rL^\infty(X,\mu)$. Let $M=q(\rL^\infty(X,\mu) \rtimes \Gamma)q$, $N=q\rL^\infty(X,\mu)$ and $\Sigma=\{ qu_gq \mid g \in \Gamma \} \subset M$. By strong ergodicity of the action, we know that every bounded $\Sigma$-central net in $N$ is trivial. Hence, by Theorem \ref{local_gap_finite}, we can find a non-zero projection $p \in N$, a finite subset $S \subset \Gamma$ and a constant $\kappa > 0$ such that for all $x \in pNp$ with $\tau(x)=0$ we have
\[ \| x \|_2 \leq \kappa \sum_{g \in S} \| x(pu_gp)-(pu_gp)x \|_2  \]
Now, take $x \in \rL^\infty(X,\mu)$ with $\tau(px)=0$. By applying the previous inequality to $px \in pNp$ we obtain
\[ \|px\|_2 \leq \kappa \sum_{g \in S} \| p(xu_g-u_gx)p \|_2 \leq \kappa \sum_{g \in S} \| p(x-u_gxu_g^*) \|_2 \]
and this is exactly the local spectral gap property with respect to $B \subset X$ where $p=\mathbf{1}_B$. 
\end{proof}

Our next application is a new proof and a generalization of \cite[Theorem 2.1]{Co75b}. Note that the original proof of \cite[Theorem 2.1]{Co75b} can be simplified so that it does not rely on singular states anymore (see \cite[Theorem 15.2.4]{PoAD}). By using Theorem \ref{local_gap_finite}, we obtain an even shorter proof which does not rely on ultraproducts.

\begin{thm} \label{connes_group} Let $N$ be a $\II_1$ factor and let $\sigma : \Gamma \curvearrowright N$ be an action of a discrete group $\Gamma$. Suppose that for every bounded net $(x_i)_{i \in I}$ in $N$ which satisfies $\lim_i \| x_i a -a x_i \|_2$ for all $a \in N$ and $\lim_i \| \sigma_g(x_i)-x_i \|_2 =0$ for all $g \in \Gamma$, we have $\lim_i \| x_i - \tau(x_i) \|_2=0$. 

Then there exists a finite set of unitaries $S \subset \mathcal{U}(N)$, a finite set $K \subset \Gamma$, and a constant $\kappa > 0$ such that for all $x \in N$ we have
\[ \|x-\tau(x)\|_2 \leq \kappa \left( \sum_{u \in S} \|ux-xu \|_2 + \sum_{g \in K} \| \sigma_g(x)-x \|_2 \right) \]
\end{thm}
\begin{proof}
Let $M=N \rtimes_\sigma \Gamma$. Let $\Sigma=\mathcal{U}(N) \cup \{u_g \mid g \in \Gamma \}$. A direct application of Theorem \ref{local_gap_finite} shows that we can find a projection $p \in N$ with $\tau(p) > \frac{1}{2}$, finite sets $S \subset \mathcal{U}(N)$ and $K \subset \Gamma$, and a constant $\kappa > 0$ such that for all $x \in pNp$ with $\tau(x)=0$ we have
\[ \|x\|_2 \leq \kappa \left( \sum_{u \in S} \|p(ux-xu)p \|_2 + \sum_{g \in K} \| p(u_gx-xu_g)p \|_2 \right) \] 
Let $v=2p-1 \in \mathcal{U}(N)$ and let $w \in \mathcal{U}(N)$ be any unitary which satisfies $w(1-p)w^* \leq p$. Let $S'=S \cup \{v,w \}$. Then it is not hard to check that there exists some constant $\kappa' > 0$ such that for all $x \in N$ we have
\[ \|x-\tau(x) \|_2 \leq \kappa' \left( \sum_{u \in S'} \|ux-xu \|_2 + \sum_{g \in K} \| \sigma_g(x)-x \|_2 \right) \]
\end{proof}

%\begin{cor}
%Let $N$ be a $\II_1$ factor and let  \sigma : \Gamma \curvearrowright N$ be an outer action of a discrete group $\Gamma$. Suppose that $N$ is not full and that $\Gamma$ is amenable. Then there exists a non-trivial bounded net $(x_i)_{i \in I}$ in $N$ which is centralizing in $N \rtimes_\sigma \Gamma$. In particular $N \rtimes_\sigma \Gamma$ is not full.
%\end{cor}
%\begin{proof}
%Since $N$ is not full, we can find a net $(p_i)_{i \in I}$ of non-zero projections $p_i \in N$ such that $\lim_i p_i =0$ and $\lim_i \frac{1}{\tau(p_i) } \|up_i-p_i u \|_1=0$ for all $u \in \mathcal{U}(N)$. Fix $\varepsilon > 0$ and $K \subset \Gamma$ a finite subset. Then, since $\Gamma$ is amenable, we can find $F \subset \Gamma$ such that $| gF \triangle F | < \varepsilon |F|$ for all $g \in K$. Let 
%\[ x_i=\left( \frac{1}{|F| \tau(p_i)}\sum_{ g \in F} \sigma_g(p_i) \right)^{1/2} \in N \]
%Note that $\|x_i \|_2=1$ for all $i$ and $\lim_i \tau(x_i) =0$. We also have 
%\[ \| ux_iu^{*}-x_i  \|_2^2 \leq \| ux_i^{2}u^{*}-x_i^{2}\|_1 \leq \frac{1}{|F|\tau(p_i)}\sum_{g \in F}\|\sigma^{-1}_g(u)p_i-p_i \sigma^{-1}_g(u) \|_1 \to 0 \] 
%Moreover, for all $g \in K$, we have
 %\[ \| \sigma_g(x_i)-x_i \|_2^2 \leq \| \sigma_g(x_i^{2})-x_i^{2} \|_1 \leq  \frac{|gF \triangle F|}{|F|}  \leq \varepsilon\]
 %Therefore, we can conclude by applying Theorem \ref{connes_group}.
%\end{proof}

We also obtain a new proof of \cite[Theorem A]{Ma16}, with a more precise conclusion regarding the choice of the state. 

\begin{thm}
Let $M$ be a full $\sigma$-finite type $\III$ factor. Then there exists a faithful normal state $\varphi$ on $M$, a finite set of positive $\varphi$-bounded elements $S \subset \rL^2(M)$ and a constant $\kappa > 0$ such that
\[ \forall x \in M, \quad \|x-\varphi(x)\|_\varphi \leq \kappa \sum_{ \xi \in S} \|x \xi - \xi x \| \]
Moreover, if $\psi$ is any given faithful normal state on $M$, we can choose $\varphi$ to be of the form $\varphi=\frac{1}{\psi(vv^*)} v^{*}\psi v$ where $v \in M$ is an isometry which is arbitrarily close to $1$.
\end{thm}
\begin{proof}
Let $\psi$ be any faithful normal state on $M$. Let $\Sigma \subset \rL^2(M)$ be the set of all positive $\psi$-bounded vectors. Note that $\Sigma$ spans a dense linear subspace of $\rL^2(M)$ (use for example the density of $\psi$-analytic elements in $M$). Hence, since $M$ is full, we know that every uniformly bounded net $(x_i)_{i \in I}$ which centralizes $\Sigma$ is trivial. Therefore, by Theorem \ref{local_gap_all}, we can find a non-zero projection $p \in M$ and a finite subset $T \subset \Sigma$ such that $pTp$ has spectral gap in $(pMp,\frac{1}{\psi(p)}p\psi p)$. But since $M$ is of type $\III$, we can find an isometry $v \in M$ such that $vv^{*}=p$. Then, it is easy to see that for $\varphi :=\frac{1}{\psi(p)}v^{*}\psi v$, the set $S=v^{*}Tv$ is $\varphi$-bounded and has spectral gap in $(M,\varphi)$. Finally, since $p$ can be chosen arbitrarily close to $1$, $v$ can also be chosen arbitrarily close to $1$.
\end{proof}

Finally, we conclude with a proposition which emphasizes the difference between local spectral gap and true spectral gap in $\II_1$ factors. Note that if $\Gamma$ is not inner amenable then $\mathcal{L}(\Gamma)$ is full \cite{Ef75} but the converse is not true \cite{Va09}.

\begin{prop}
Let $\Gamma$ be an i.c.c group. Let $M=\mathcal{L}(\Gamma)$ be the associated $\II_1$ factor. If $M$ is full then we can find a projection $p \in M$, arbitrarily close to $1$, a finite set $K \subset \Gamma$ and a constant $\kappa > 0$ such that for all $x \in pMp$ with $\tau(x)=0$ we have
\[ \| x \|_2 \leq \kappa \sum_{g \in K} \| p(xu_g-u_gx)p \|_2 \]
Moreover, one can choose $p=1$ if and only if $\Gamma$ is not inner amenable.
\end{prop}
\begin{proof}
It is a direct application of Theorem \ref{local_gap_finite} with $\Sigma=\{ u_g \mid g \in \Gamma \}$.
\end{proof}

\bibliographystyle{plain}

\begin{thebibliography}{MvN43}

\bibitem[BISG15]{BISG15} { \sc R. Boutonnet, A. Ioana, A. Salehi Golsefidy}, {\it Local spectral gap in simple Lie groups and applications.} {To appear in Invent. Math.} {\tt arXiv:1503.06473}

%\bibitem[AH12]{AH12} {\sc H. Ando, U. Haagerup}, {\it Ultraproducts of von Neumann algebras.} J. Funct. Anal. {\bf 266} (2014), 6842--6913.


%\bibitem[Ba93]{Ba93} {\sc L. Barnett}, {\it Free product von Neumann algebras of type ${\rm III}$.} Proc. Amer. Math. Soc. {\bf 123} (1995), 543--553.

%\bibitem[Bi89]{Bi89} {\sc D. Bisch}, {\it On the existence of central sequences in subfactors.} Trans. Amer. Math. Soc. {\bf 321} (1989), 117--128.

%\bibitem[BH16]{BH16} {\sc R. Boutonnet, C. Houdayer}, {\it Amenable absorption in amalgamated free product von Neumann algebras.} {To appear in Kyoto J. Math.} {\tt arXiv:1606.00808}



%\bibitem[Co72]{Co72} {\sc A. Connes}, {\it Une classification des facteurs de type ${\rm III}$.} Ann. Sci. \'{E}cole Norm. Sup. {\bf 6} (1973), 133--252.

%\bibitem[Co74]{Co74} {\sc A. Connes}, {\it Almost periodic states and factors of type ${\rm III_1}$.} J. Funct. Anal. {\bf 16} (1974), 415--445.

%\bibitem[Co75a]{Co75a} {\sc A. Connes}, {\it Outer conjugacy classes of automorphisms of factors.} Ann. Sci. \'{E}cole Norm. Sup. {\bf 8} (1975), 383--419.

\bibitem[Co75b]{Co75b} {\sc A. Connes}, {\it Classification of injective factors. Cases ${\rm II_1}$, ${\rm II_\infty}$, ${\rm III_\lambda}$, $\lambda \neq 1$.} Ann. of Math. {\bf 74} (1976), 73--115.

\bibitem[CS78]{CS78} {\sc A. Connes, E. Stormer}, {\it Homogeneity of the state space of factors of type $\III_1$.} J. Funct. Anal. {\bf 28} (1978), 187--196.

\bibitem[Va09]{Va09} {\sc S. Vaes}, {\it An inner amenable group whose von Neumann algebra
does not have property Gamma.} Acta Math. {\bf 208} (2009), 389--394.

\bibitem[Ef75]{Ef75} {\sc S. Vaes}, {\it Property $\Gamma$ and inner amenability.} Proc. Amer. Math. Soc. {\bf 47} (1975), 483--486.

%\bibitem[Co85]{Co85} {\sc A. Connes}, {\it Factors of type ${\rm III_1}$, property $L'_\lambda$ and closure of inner automorphisms.} J. Operator Theory {\bf 14} (1985), 189--211.

%\bibitem[CT76]{CT76} {\sc A. Connes, M. Takesaki}, {\it The flow of weights of factors of type ${\rm III}$.} Tohoku Math. J. {\bf 29} (1977), 473--575. 

%\bibitem[Ha73]{Ha73} {\sc U. Haagerup}, {\it The standard form of von Neumann algebras.} Math. Scand. {\bf 37} (1975), 271--283.

%\bibitem[Ha85]{Ha85} {\sc U. Haagerup}, {\it Connes' bicentralizer problem and uniqueness of the injective factor of type ${\rm III_1}$.} Acta Math. {\bf 69} (1986), 95--148.

%\bibitem[Ho15]{Ho15} {\sc D.J. Hoff}, {\it Von Neumann algebras of equivalence relations with nontrivial one-cohomology.} J. Funct. Anal. {\bf 270} (2016), 1501--1536.

%\bibitem[Ho06]{Ho06} {\sc C. Houdayer}, {\it A new construction of factors of type ${\rm III_1}$.} J. Funct. Anal. {\bf 242} (2007), 375--399.

%\bibitem[Ho14]{Ho14} {\sc C. Houdayer}, {\it Gamma stability in free product von Neumann algebras.}  Commun. Math. Phys. {\bf 336} (2015), 831--851.

%\bibitem[HI15]{HI15} {\sc C. Houdayer, Y. Isono}, {\it Unique prime factorization and bicentralizer problem for a class of type ${\rm III}$ factors.} Adv. Math. {\bf 305} (2017), 402--455.

%\bibitem[HI15b]{HI15b} {\sc C. Houdayer, Y. Isono}, {\it Bi-exact groups, strongly ergodic actions and group measure space type ${\rm III}$ factors with no central sequence.} Comm. Math. Phys. {\bf 348} (2016), 991--1015.

%\bibitem[HR14]{HR14} {\sc C. Houdayer, S. Raum}, {\it Asymptotic structure of free Araki--Woods factors.} Math. Ann. {\bf 363} (2015), 237--267.

%\bibitem[HU15a]{HU15a} {\sc C. Houdayer, Y. Ueda}, {\it Asymptotic structure of free product von Neumann algebras.} \\ {\tt arXiv:1503:02460}.

%\bibitem[HU15]{HU15} {\sc C. Houdayer, Y. Ueda}, {\it Rigidity of free product von Neumann algebras.} Compos. Math. {\bf 152} (2016), 2461--2492.

%\bibitem[Is14]{Is14} {\sc Y. Isono}, {\it Some prime factorization results for free quantum group factors.} To appear in J. Reine Angew. Math. {\tt arXiv:1401.6923}

%\bibitem[Jo81]{Jo81} {\sc V.F.R. Jones}, {\it Central sequences in crossed products of full factors.} Duke Math. J. {\bf 49} (1982), 29--33.

%\bibitem[KR97]{KR97} {\sc R.V. Kadison, J.R. Ringrose}, {\it Fundamentals of the theory of operator algebras. ${\rm I}$. Elementary theory.} Reprint of the 1983 original. Graduate Studies in Mathematics, {\bf 15}. American Mathematical Society, Providence, RI, 1997. xvi+398 pp.

%\bibitem[McD69]{McD69} {\sc D. McDuff}, {\it Central sequences and the hyperfinite factor.} Proc. London Math. Soc. {\bf 21} (1970), 443--461.

%\bibitem[Ma16a]{Ma16a} {\sc A. Marrakchi}, {\it Solidity of type $\mathrm{III}$ Bernoulli crossed products.} To appear in Comm. Math. Phys. {\tt arXiv:1601.03666}

\bibitem[Ma16]{Ma16} {\sc A. Marrakchi}, {\it Spectral gap characterization of full type $\mathrm{III}$ factors.} {To appear in J. Reine Angew. Math.} {\tt arXiv:1605.09613}

%\bibitem[MvN43]{MvN43} {\sc F. Murray, J. von Neumann}, {\it Rings of operators.} ${\rm IV}$.  Ann. of Math. {\bf 44} (1943), 716--808.

%\bibitem[Oc85]{Oc85} {\sc A. Ocneanu}, {\it Actions of discrete amenable groups on von Neumann algebras.} Lecture Notes in Mathematics, {\bf 1138}. Springer-Verlag, Berlin, 1985. iv+115 pp.
%

%\bibitem[Oz03]{Oz03} {\sc N. Ozawa}, {\it Solid von Neumann algebras.} Acta Math.\ {\bf 192} (2004), 111--117.

%\bibitem[Po01]{Po01} {\sc S. Popa}, {\it On a class of type ${\rm II_1}$ factors with Betti numbers invariants.} Ann. of Math. {\bf 163} (2006), 809--899.

%\bibitem[Po03]{Po03} {\sc S. Popa}, {\it Strong rigidity of ${\rm II_1}$ factors arising from malleable actions of w-rigid groups $\rm I$.} Invent. Math. {\bf 165} (2006), 369--408.

%\bibitem[Po06]{Po06} {\sc S. Popa}, {\it On Ozawa's property for free group factors.} Int. Math. Res. Not. IMRN {\bf 2007}, no. 11, Art. ID rnm036, 10 pp.

\bibitem[PoAD]{PoAD} {\sc S. Popa, C. Anantharamn-Delaroche}, {\it An introduction to ${\rm II_1}$ factors. } (book in progress) {\tt http://www.math.ucla.edu/~popa/Books/IIun-v13.pdf}

%\bibitem[Ra99]{Ra99} {\sc Y. Raynaud}, {\it On ultrapowers of non commutative $\rL_p$-spaces.} J. Operator Theory {\bf 48} (2002), 41--68.

%\bibitem[Sh96]{Sh96} {\sc D. Shlyakhtenko}, {\it Free quasi-free states.} Pacific J. Math. {\bf 177} (1997), 329--368. 

%\bibitem[Ta03]{Ta03} {\sc M. Takesaki}, {\it Theory of operator algebras. ${\rm II}$.} Encyclopaedia of Mathematical Sciences, {\bf 125}. Operator Algebras and Non-commutative Geometry, 6. Springer-Verlag, Berlin, 2003. xxii+518 pp.

%
%\bibitem[Ta03b]{Ta03b} {\sc M. Takesaki}, {\it Theory of operator algebras. ${\rm III}$.} Encyclopaedia of Mathematical Sciences, {\bf 127}. Operator Algebras and Non-commutative Geometry, 8. Springer-Verlag, Berlin, 2003. xxii+548 pp.
%

%\bibitem[TU14]{TU14} {\sc R. Tomatsu, Y. Ueda}, {\it A characterization of fullness of continuous cores of type ${\rm III_1}$ free product factors.} Kyoto J. Math. {\bf 56} (2016), 599--610. 

%\bibitem[Ue10]{Ue10} {\sc Y. Ueda}, {\it Factoriality, type classification and fullness for free product von Neumann algebras.} Adv. Math. {\bf 228} (2011), 2647--2671.

%\bibitem[Va04]{Va04} {\sc S. Vaes}, {\it ƒ\'Etats quasi-libres libres et facteurs de type ${\rm III}$ (d'apr\`es D.\ Shlyakhtenko).} SŽ\'eminaire Bourbaki, expos\'eŽ 937, AstŽ\'erisque {\bf 299} (2005), 329--350.

%\bibitem[VV14]{VV14} {\sc S. Vaes, P. Verraedt}, {\it Classification of type ${\rm III}$ Bernoulli crossed products.} Adv. Math. {\bf 281} (2015), 296--332.

\end{thebibliography}

\end{document}